\documentclass[12pt]{elsarticle}
\usepackage{amssymb}
\usepackage{amsthm}
\usepackage{enumerate}
\usepackage{amsmath,amssymb,amsfonts,amsthm,fancyhdr}
\usepackage{mathrsfs}
\usepackage[colorlinks,linkcolor=black,anchorcolor=black,citecolor=black,CJKbookmarks=True]{hyperref}
\usepackage{tikz}
\usepackage{mathabx}

\newdefinition{rem}{Remark}[section]
\newdefinition{theorem}{Theorem}[section]
\newdefinition{corollary}{Corollary}[section]
\newdefinition{definition}{Definition}[section]
\newdefinition{lemma}{Lemma}[section]
\newdefinition{prop}{Proposition}[section]
\numberwithin{equation}{section}

\begin{document}
\begin{frontmatter}

\title{Some exceptional sets of Borel-Bernstein Theorem in continued fractions}

\author[a]{Lulu Fang}\ead{fanglulu1230@gmail.com}
\author[b]{Jihua Ma}\ead {jhma@whu.edu.cn}
\author[b,c]{Kunkun Song}\ead{songkunkun@whu.edu.cn}
\address[a]{School of Mathematics, Sun Yat-sen University, Guangzhou, 510275, P.R.China}
\address[b]{School of Mathematics and Statistics, Wuhan University, Wuhan, 430072, P.R.China}
\address[c]{Universit\'{e} Paris-Est, LAMA (UMR 8050), UPEMLV, UPEC, CNRS, F-94010, Cr\'{e}teil, France}

\begin{abstract}\par
Let $[a_1(x),a_2(x), a_3(x),\cdots]$ denote the continued fraction expansion of a real number $x \in [0,1)$.
This paper is concerned  with certain exceptional sets of the Borel-Bernstein Theorem on the growth rate of $\{a_n(x)\}_{n\geq1}$. As a main result, the Hausdorff dimension of the set
\[
E_{\sup}(\psi)=\left\{x\in[0,1):\ \limsup\limits_{n\to\infty}\frac{\log a_n(x)}{\psi(n)}=1\right\}
\]
 is determined, where $\psi:\mathbb{N}\rightarrow\mathbb{R}^+$ tends to infinity  as $n\to\infty$.
\end{abstract}


\end{frontmatter}

\section{Introduction}
It is well known that each real number $x \in [0,1)$ admits a \emph{continued fraction expansion} of the form
\begin{equation}\label{cfe}
x = \dfrac{1}{a_1(x) +\dfrac{1}{a_2(x)  +\dfrac{1}{a_3(x)+\ddots}}}
\end{equation}
where $a_1(x),a_2(x),a_3(x),\cdots$ are positive integers, called the \emph{partial quotients} of the continued fraction expansion of $x$. For simplicity, we denote \eqref{cfe} by $x=[a_1(x), a_2(x), a_3(x),\cdots]$. Such an $x$ is irrational if and only if \eqref{cfe} is infinite. For more details about continued fractions, see \cite{IK02, Khi64}.\\
\indent Continued fractions are closely allied to the theory of Diophantine approximation.
The classical theorem of Dirichlet implies that for any real number $x$, there exist infinitely many ``good" rational approximates $p/q\ (q>0)$ such that
$\left|x-p/q\right| < 1/q^{2}$. Continued fractions provide a simple mechanism for generating these rational approximates.
On the contrary, there are some irrationals, called \emph{badly approximable numbers}, such that $\left|x-p/q\right| \geq c_x/q^{2}$ holds for all rationals $p/q\ (q>0)$, where $c_x$ is a positive constant depending only on $x$. Via continued fractions, badly approximable numbers have a beautiful characterisation: a number is badly approximable if and only if its partial quotients are bounded.
Let us recall Liouville's theorem on Diophantine approximation  which implies the transcendence of real numbers with rapidly increasing partial quotients.
All the above results lead to the study of
the growth rate of partial quotients.\\
\indent Borel \cite{BE1909} proved that for Lebesgue almost all $x \in [0,1)$, $\limsup_{n \to \infty} a_n(x) =\infty$. Equivalently,
all badly approximable numbers form a null set with respect to Lebesgue measure.
In fact, a more precise result is the  Borel-Bernstein Theorem (``0--1" law), see \cite{BF1912,BE1909,BE1912}, which
asserts that for Lebesgue almost all $x \in [0,1)$, $a_n(x) \geq \psi(n)$ holds for infinitely many $n'$s or finitely many $n'$s depending on whether $\sum_{n \geq 1} 1/\psi(n)$ diverges or converges.
This naturally leads one to investigate the sizes of the exceptional sets related to  the growth rate of partial quotients in the sense of Hausdorff dimension. The first published work was due to Jarn\'{\i}k \cite{Jar28} who proved that the set of real numbers whose partial quotients are bounded has full Hausdorff dimension.
Later on, Good \cite{Good41} showed that the set of $x\in [0,1)$ in which $a_n(x) \to \infty$ as $n \to \infty$ is of Hausdorff dimension $1/2$. After that, there are many papers studying
the growth rate of partial quotients from various aspects (e.g., sum, maxima, etc), see for example, Hirst \cite{lesHir73}, Ramharter \cite{R}, Cusick \cite{Cus90}, {\L}uczak \cite{lesLuc97}, Wang and Wu \cite{lesWW08A, WW08}, Wu and Xu \cite{WX09}, Xu \cite{XU08}, Jordan and Rams \cite{lesJR12}, Fan, Liao, Wang and Wu \cite{FLWW09, FLWW13}, Cao, Wang and Wu \cite{CWW}, Liao and Rams \cite{LR16,LR161}, Fang and Song  \cite{FS16,FS19}.

 As a consequence of the Borel-Bernstein Theorem, for Lebesgue almost all $x \in [0,1)$,
\begin{equation}\label{limsupan}
\limsup\limits_{n\to\infty}\frac{\log a_n(x)}{\log n}=1.
\end{equation}
This means that for almost all $x\in[0,1)$, there exists a subsequence of partial quotients tends to infinity nearly with a linear speed.
This paper is concerned with Hausdorff dimension of some exceptional sets of \eqref{limsupan}.
 More precisely, we consider dimensions of the set
\[E_{\sup}(\psi)=\left\{x\in[0,1):\ \limsup\limits_{n\to\infty}\frac{\log a_n(x)}{\psi(n)}=1\right\},\]
where $\psi:\mathbb{N}\rightarrow\mathbb{R}^+$ is a function satisfying $\psi(n)\to\infty$ as $n\to\infty$. In the sequel, we use the notation $\dim_{\rm H}$ to denote the Hausdorff dimension. Our main results are as follows.
%

\begin{theorem}\label{zdl}
Let $\psi:\mathbb{N}\rightarrow\mathbb{R}^+$ be a function satisfying $\psi(n)\to\infty$ as $n\to\infty$. We have
\begin{enumerate}[(i)]
\item if $\psi(n)/n\to0$ as $n\to\infty$, then $\dim_{\rm H}E_{\sup}(\psi)=1$,
\item if $\psi(n)/n\to\alpha\ (0<\alpha<\infty)$ as $n\to\infty$, then $\dim_{\rm H}E_{\sup}(\psi)=S(\alpha)$,
\item if $\psi(n)/n\to\infty$ as $n\to\infty$, then $\dim_{\rm H}E_{\sup}(\psi)=\frac{1}{A+1}$,
\end{enumerate}
where $A$ is given by
\begin{equation}\label{phitj}
\log A=\liminf\limits_{n\to\infty}\frac{\log\psi(n)}{n},
\end{equation}
and $S:\mathbb{R}^+\rightarrow[\frac{1}{2},1]$ is a decreasing continuous function satisfying
\[\lim\limits_{\alpha\to0}S(\alpha)=1\ \ \text{and}\ \ \lim\limits_{\alpha\to\infty}S(\alpha)=\frac{1}{2}.\]
\end{theorem}
\begin{rem}\label{sb}
We do not have an explicit formula for $S(\alpha)$. In fact, it is exactly
\[S(\alpha)=\dim_{\rm H}\{x\in[0,1):\ a_n(x)\geq(\exp(\alpha))^{n}\ \text{for infinity many}\ n\}.\]
For more details, see Lemma 2.6 and Theorem 3.1 in Wang and Wu \cite{WW08}.
\end{rem}

Besides the set $E_{\sup}(\psi)$, we are also interested in the Hausdorff dimension of the sets
\[E(\psi)=\left\{x\in[0,1): \lim\limits_{n\to\infty}\frac{\log a_n(x)}{\psi(n)}=1\right\}\ \text{and}\ E_{\inf}(\psi)=\left\{x\in[0,1):\ \liminf\limits_{n\to\infty}\frac{\log a_n(x)}{\psi(n)}=1\right\}.\]

\begin{theorem}\label{jxybqk}
Let $\psi:\mathbb{N}\rightarrow\mathbb{R}^+$ be a function satisfying $\psi(n)\to\infty$ as $n\to\infty$.
Then
\begin{equation*}
\dim_{\rm H}E(\psi)=
\frac{1}{2+\xi},
\end{equation*}
where $\xi$ is defined as
\[
\xi:=\limsup\limits_{n\to\infty}\frac{\psi(n+1)}{\psi(1)+\cdots+\psi(n)}.
\]

\end{theorem}

\begin{theorem}\label{xjx}
Let $\psi:\mathbb{N}\rightarrow\mathbb{R^{+}}$ be a function satisfying $\psi(n)\to\infty$ as $n\to\infty$. Then
\[
\dim_{\rm H}E_{\inf}(\psi)=\frac{1}{B+1},
 \]
 where $B$ is given by
 \[
\log B:=\limsup\limits_{n\to\infty}\frac{\log\psi(n)}{n}.
 \]
\end{theorem}
\begin{rem}\label{xd}
Under the condition that $\psi(n)\to \infty$ as $n \to \infty$, we  always have
\[\dim_{\rm H}E(\psi)\leq\dim_{\rm H}E_{\inf}(\psi)\leq\dim_{\rm H}E_{\sup}(\psi).\]
More precisely, we divide it into three parts.
\begin{enumerate}[(i)]
\item $\psi(n)/n\to0$ as $n\to\infty$,
\[\dim_{\rm H}E(\psi)\leq\dim_{\rm H}E_{\inf}(\psi)=\frac{1}{2}<\dim_{\rm H}E_{\sup}(\psi)=1.\]
\item $\psi(n)/n\to \alpha\ (\alpha>0)$ as $n\to\infty$,
\[\dim_{\rm H}E(\psi)=\dim_{\rm H}E_{\inf}(\psi)=\frac{1}{2}\leq\dim_{\rm H}E_{\sup}(\psi)\leq1.\]
\item $\psi(n)/n\to\infty$ as $n\to\infty$,
\[\dim_{\rm H}E(\psi)\leq\dim_{\rm H}E_{\inf}(\psi)\leq\dim_{\rm H}E_{\sup}(\psi)\leq\frac{1}{2}.\]
\end{enumerate}
\end{rem}
\begin{rem}\label{ks}
Let $1<a<b<\infty$ and
\begin{equation*}
\psi(n)=
\begin{cases}
a^n,\ \ \ \ \ \ \text{if}\ \ n\ \text{is even};\cr
b^n,\ \ \ \ \ \ \text{if}\ \ n\ \text{is odd}.
\end{cases}
\end{equation*}
Then $A=a,B=b$ and
\begin{align*}
\xi=&\limsup\limits_{n\to\infty}\frac{\psi(n+1)}{\psi(1)+\cdots+\psi(n)}=
\lim\limits_{k\to\infty}\frac{\psi(2k+1)}{\psi(1)+\cdots+\psi(2k)}\\
=&\lim\limits_{k\to\infty}\frac{b^{2k+1}}{a^{2}+\cdots+a^{2k}+b^{1}+\cdots+b^{2k-1}}
=b^2-1.
\end{align*}
Hence we have
\[\frac{1}{b^2+1}=\dim_{\rm H}E(\psi)<\frac{1}{b+1}=\dim_{\rm H}E_{\inf}(\psi)<\frac{1}{a+1}=\dim_{\rm H}E_{\sup}(\psi).\]
This implies that the dimensional result changes essentially when we replace $\lim$ by $\liminf$ or $\limsup$. The phenomenon also happens in Liao and Rams \cite{LR16}.
\end{rem}
\ \ \ \ We use $\mathbb{N}$ to denote the set of all positive integers, $|\cdot|$ the length of a subset of $[0,1)$, $\exp(x)$ the natural exponential function, $\lfloor x\rfloor$ the largest integer not exceeding $x$ and $\sharp$ the cardinality of a set, respectively.\\
\indent The paper is organized as follows. In section 2, we present some elementary properties and dimensional results in continued fractions. Section 3 is devoted to the proofs of the main results.

\section{Preliminaries}
\subsection{Elementary properties of continued fractions}
 For any $n\geq1$ and $(a_1,\cdots,a_n)\in\mathbb{N}^{n}$, we call
\begin{equation*}
I_{n}(a_1, \cdots, a_n): =\left\{x\in[0,1):\ a_1(x)=a_1, \cdots, a_n(x)=a_n\right\}
\end{equation*}
 a cylinder of order $n$ of continued fractions and denote the $n$-th convergent of the continued fraction expansion of $x$ by
\begin{equation*}
\frac{p_n(x)}{q_n(x)}:=[a_1(x),a_2(x),\cdots,a_n(x)].
\end{equation*}
Notice that all points in $I_{n}(a_1, \cdots, a_n)$ have the same $p_n(x)$ and $q_n(x)$. Thus, we write
\begin{equation*}
p_n(a_1,\cdots,a_n)=p_n=p_n(x)\ \text{and}\ q_n(a_1,\cdots,a_n)=q_n=q_n(x)
\end{equation*}
for $x\in I_{n}(a_1, \cdots, a_n)$.
It is well known (see \cite[p. 4]{Khi64}) that $p_n$ and $q_n$ satisfy the following recursive formula:
\begin{equation}\label{ppqq}
\begin{cases}
p_{-1}=1,\ \ p_0=0,\ \ p_n=a_np_{n-1}+p_{n-2}\ (n\geq1);\cr
q_{-1}=0,\ \ \ q_0=1,\ \ q_n=a_nq_{n-1}+q_{n-2}\ (n\geq1).
\end{cases}
\end{equation}
As consequences, we have the following results.
\begin{lemma}[\cite{Khi64}]\label{fl}
For any $n\geq1$, we have
\[q_n\geq2^{\frac{n-1}{2}}.\]
\end{lemma}

\begin{lemma}[\cite{Wu06}]\label{qan}
For any $n\geq1$ and $1\leq k\leq n$,
\[\frac{a_k+1}{2}\leq\frac{q_{n}(a_1,\cdots, a_n)}{q_{n-1}(a_1,\cdots, a_{k-1}, a_{k+1},\cdots,a_n)}\leq a_k+1.\]
\end{lemma}

\begin{prop}[{\cite[p. 18]{IK02}}]\label{cd}
For any $(a_1,\cdots, a_n)\in\mathbb{N}^{n}$, the cylinder $I_{n}(a_1,\cdots, a_n)$ is the interval with the endpoints
$p_n/q_n$ and $(p_n+p_{n-1})/(q_n+q_{n-1})$.
As a result, the length of $I_{n}(a_1, \cdots, a_n)$ equals to
\begin{equation*}
|I_{n}(a_1, \cdots, a_n)|=\frac{1}{q_n(q_n+q_{n-1})}.
\end{equation*}
\end{prop}

\subsection{Some useful lemmas to estimate Hausdorff dimension}
For any $M\in\mathbb{N}$, let $E_M$ be the set of points in $[0,1)$ whose partial quotients in continued fraction expansion do not exceed $M$. That is,
\[E_{M}=\{x\in[0,1):\ 1\leq a_n(x)\leq M, \forall\  n\geq1\}.\]
Jarn\'{\i}k  \cite{Jar28} considered its Hausdorff dimension.
\begin{lemma}\label{Ja}
For any $M\geq8$,
\[1-\frac{1}{M\log2}\leq\dim_{\rm H}E_{M}\leq 1-\frac{1}{8M\log M}.\]
\end{lemma}
The following dimensional result is useful for  obtaining  the lower bound estimate of Hausdorff dimension for some  sets in continued fractions.
\begin{lemma}[{\cite[Lemma 3.2]{FLWW09}}]\label{flww}
Let $\{t_n\}_{n\geq1}$ be a sequence of positive integers tending to infinity with $t_n\geq3$ for all $n\geq1$. Then for any positive number $N\geq2$,
\[\dim_{\rm H}\{x\in[0,1): t_n\leq a_n(x)<Nt_n, \forall\ n\geq1\}=
\liminf\limits_{n\to\infty}\frac{\log(t_1t_2\cdots t_n)}{2\log(t_1t_2\cdots  t_n)+\log t_{n+1}}.\]
\end{lemma}

\section{Proofs of main results}
In this section, we give the proofs of main results. Our proofs are inspired by  Wu and Xu \cite{WX09}, Wang and Wu \cite{WW08} and especially by Liao and Rams \cite{LR16}.
\subsection{The proof of Theorem \ref{zdl}}
 We divide the proof into three cases. Recall that
\[E_{\sup}(\psi)=\left\{x\in[0,1):\ \limsup\limits_{n\to\infty}\frac{\log a_n(x)}{\psi(n)}=1\right\}.\]
\subsubsection{\textbf{The case} $\psi(n)/n\to0$ as $n\to\infty$.}
 In this case, our strategy is to construct a suitable Cantor subset $E_{M}(\psi)$ of $E_{\sup}(\psi)$ and then establish a connection between $E_{M}(\psi)$ and $E_M$ by means of a $(1+\varepsilon)$ H\"{o}lder function. Choose a strictly increasing subsequence $\{m_k\}_{k\geq1}\subseteq\mathbb{N}$ satisfying $m_k=2^k$  for any $k\geq1$. For any $M\in\mathbb{N}$, define
\begin{align*}
  E_M(\psi)=\Big\{x\in[0,1): a_{m_k}(x)=\lfloor\exp\psi(m_k)\rfloor\ \text{and}\ 1\leq a_i(x)\leq M~(i\neq m_k),\ \forall\ k\geq1\Big\}.
\end{align*}

\begin{lemma}\label{phibh}
For any $M\in\mathbb{N}$, $E_M(\psi)\subseteq E_{\sup}(\psi)$.
\end{lemma}
\begin{proof}
Notice that $\psi(n)\to\infty$ as $n\to\infty$, then for any positive number $c$, there exists a positive integer $N$ such that for any $k\geq N$, $\psi(k)>c$. Take $c=\log2M$, then for any $m_k>k\geq N$, we have
$\lfloor\exp\psi(m_k)\rfloor\geq\lfloor\exp(\log2M)\rfloor\geq M$. Hence, for any $x\in E_M(\psi)$,
\[\lim\limits_{k\to\infty}\frac{\log a_{m_k}(x)}{\psi(m_k)}=1,\ \text{i.e.},\ \limsup\limits_{n\to\infty}\frac{\log a_{n}(x)}{\psi(n)}=1.\]
we complete the proof.
\end{proof}
To estimate the Hausdorff dimension of $E_M(\psi)$, in the following we shall make use of a kind of symbolic space described below.
For any $n\geq1$, set
\[C_n=\left\{(\sigma_1,\cdots,\sigma_n)\in\mathbb{N}^{n}:\sigma_{m_k}=\lfloor\exp\psi(m_k)\rfloor\ \text{and}\ 1\leq\sigma_i\leq M,1\leq i\neq m_k\leq n\right\}.\]
For any $n\geq1$ and $(\sigma_1,\cdots,\sigma_n)\in C_n$, we call
$I_n(\sigma_1,\cdots,\sigma_n)$ the cylinder of order $n$ and \[J_n(\sigma_1,\cdots,\sigma_n)=\bigcup\limits_{\sigma_{n+1}}I_{n+1}(\sigma_1,\cdots,\sigma_n,\sigma_{n+1})\]
a fundamental interval of order $n$,
where $(\sigma_1,\cdots,\sigma_n,\sigma_{n+1})\in C_{n+1}$. It is obvious that
\[E_M(\psi)=\bigcap\limits_{n\geq1}\bigcup\limits_{(\sigma_1,\cdots,\sigma_n)\in C_n}I_n(\sigma_1,\cdots,\sigma_n)=
\bigcap\limits_{n\geq1}\bigcup\limits_{(\sigma_1,\cdots,\sigma_n)\in C_n}J_n(\sigma_1,\cdots,\sigma_n).\]
Let $t(n)=\#\{k:2^k\leq n\}$, clearly we have
\begin{equation}\label{jbo}
\lim\limits_{n\to\infty}\frac{t(n)}{n}=0.
\end{equation}
 For any $(\sigma_1,\cdots,\sigma_n)\in C_n$, let $\overline{(\sigma_1,\cdots,\sigma_n)}$ be the block
obtained by eliminating the terms $\{\sigma_{m_k}:1\leq k\leq t(n)\}$ in $(\sigma_1,\cdots,\sigma_n)$, then we can write
\[\overline{(\sigma_1,\cdots,\sigma_n)}\in\mathcal\{1, 2,\cdots, M\}^{n-t(n)}.\]
For simplicity, we set
\begin{equation}\label{iq}
\overline{I_n}(\sigma_1,\cdots,\sigma_n)=I_{n-t(n)}(\sigma_1,\cdots,\sigma_n),\ \overline{q_n}(\sigma_1,\cdots,\sigma_n)=q_{n-t(n)}(\sigma_1,\cdots,\sigma_n).
\end{equation}
It is worth to note that $\lim\limits_{n\to\infty}\psi(n)/n=0$ and $m_k=2^k\ (k\geq1)$, then
\begin{align*}
0\leq\limsup\limits_{n\to\infty}\frac{1}{n}\sum\limits_{i=1}^{t(n)}\psi(m_i)
=&\limsup\limits_{n\to\infty}\frac{1}{n}\left(\psi(m_1)+\psi(m_2)+\cdots+\psi(m_{t(n)})\right)\\
\leq&\limsup\limits_{n\to\infty}\frac{1}{2^{t(n)}}\left(\psi(2^1)+\psi(2^2)+\cdots+\psi(2^{t(n)})\right)\\
\leq&\limsup\limits_{n\to\infty}\frac{1}{2^{n}}\left(\psi(2^1)+\psi(2^2)+\cdots+\psi(2^{n})\right)\\
\leq&\limsup\limits_{n\to\infty}\frac{\psi(2^{n+1})}{2^{n+1}-2^n}=
2\limsup\limits_{n\to\infty}\frac{\psi(2^{n+1})}{2^{n+1}}=0,
\end{align*}
thus we have
\begin{equation}\label{zytj}
\limsup\limits_{n\to\infty}\frac{1}{n}\sum\limits_{i=1}^{t(n)}\psi(m_i)=0.
\end{equation}
 \begin{lemma}\label{qjbj}
 For any $\varepsilon>0$, there exists $N_1$ such that for any $n\geq N_1$ and $(\sigma_1,\cdots,\sigma_n)\in C_n$,
 \[|I_n(\sigma_1,\cdots,\sigma_n)|\geq|\overline{I_n}(\sigma_1,\cdots,\sigma_n)|^{1+\varepsilon}.\]
 \end{lemma}
 \begin{proof}
 Let $\varepsilon>0$, it follows \eqref{jbo}, \eqref{iq}, \eqref{zytj} and Lemma \ref{fl} that there exists $N_1$ such that for any $n\geq N_1$ and $(\sigma_1,\cdots,\sigma_n)\in C_n$,
 \begin{equation}\label{zy}
 \overline{q_n}^{2\varepsilon}(\sigma_1,\cdots,\sigma_n)
 =q_{n-t(n)}^{2\varepsilon}(\sigma_1,\cdots,\sigma_n)\geq2^{(n-t(n)-1)\varepsilon}
 \geq2\prod\limits_{i=1}^{t(n)}(\exp\psi(m_i)+1)^{2}.
 \end{equation}
 In view of \eqref{zy}, by Lemma \ref{qan} and Proposition \ref{cd}, we have
 \begin{eqnarray*}
 |I_n(\sigma_1,\cdots,\sigma_n)|&\geq&\frac{1}{2q^{2}_n(\sigma_1,\cdots,\sigma_n)}
 \geq\frac{1}{2\left(q_{n-t(n)}(\sigma_1,\cdots,\sigma_n)\prod\limits_{i=1}^{t(n)}(\exp\psi(m_i)+1)\right)^{2}}\\
 &\geq&\frac{1}{q_{n-t(n)}^{2+2\varepsilon}(\sigma_1,\cdots,\sigma_n)}
 \geq|\overline{I_n}(\sigma_1,\cdots,\sigma_n)|^{1+\varepsilon}.
 \end{eqnarray*}
\end{proof}
For any $x=[\eta_1,\eta_2,\cdots]\in E_M(\psi), y\in[\tau_1,\tau_2,\cdots]\in E_M(\psi)$ and $x\neq y$, then there exists the greatest integer $n$ such that $x,y$ are contained in the same the cylinder of order $n$, i.e., $x,y\in I_n(\sigma_1,\cdots,\sigma_n)$. Therefore  there exist
 $l_{n+1}\neq r_{n+1}$ such that $(\sigma_1,\cdots,\sigma_n,l_{n+1})\in C_{n+1}$, $(\sigma_1,\cdots,\sigma_n,r_{n+1})\in C_{n+1}$ and
 $x\in I_{n+1}(\sigma_1,\cdots,\sigma_n,l_{n+1}), y\in I_{n+1}(\sigma_1,\cdots,\sigma_n,r_{n+1})$ respectively. We next compare $|x-y|$ with $|I_{n}(\sigma_1,\cdots,\sigma_n)|$.
 Since
 \[I_{n+1}(\sigma_1,\cdots,\sigma_n,l_{n+1})\bigcap E_M(\psi)=J_{n+1}(\sigma_1,\cdots,\sigma_n,l_{n+1})\bigcap E_M(\psi),\]
 \[I_{n+1}(\sigma_1,\cdots,\sigma_n,r_{n+1})\bigcap E_M(\psi)=J_{n+1}(\sigma_1,\cdots,\sigma_n,r_{n+1})\bigcap E_M(\psi),\]
 we have
 $x\in J_{n+1}(\sigma_1,\cdots,\sigma_n,l_{n+1}), y\in J_{n+1}(\sigma_1,\cdots,\sigma_n,r_{n+1})$.
 As a consequence, $|y-x|$ is not less than the gap between $J_{n+1}(\sigma_1,\cdots,\sigma_n,r_{n+1})$ and $J_{n+1}(\sigma_1,\cdots,\sigma_n,l_{n+1})$.
 \begin{lemma}\label{jg}
\begin{equation*}
|x-y|\geq\frac{1}{(M+2)^3}|I_n(\sigma_1,\cdots,\sigma_n)|.
\end{equation*}
\end{lemma}
\begin{proof}
Without loss of generality, we assume that $x<y$ and $n$ is even, for others cases, the proofs are similar. Notice that $n+1\neq m_k$ for any $k\in\mathbb{N}$, otherwise $l_{n+1}=r_{n+1}=\lfloor\exp\psi(m_k)\rfloor$. So we divide the proof into two parts.
\begin{enumerate}[(i)]
\item If $n+2=m_k$ for some $k\geq1$, let $\lfloor\exp\psi(m_k)\rfloor=t$, then $t\geq1$ and we have
\begin{align*}
|y-x|=&\left|\frac{(r_{n+1}+\frac{1}{t})p_n+p_{n-1}}{(r_{n+1}+\frac{1}{t})q_n+q_{n-1}}-
\frac{(l_{n+1}+\frac{1}{t+1})p_n+p_{n-1}}{(l_{n+1}+\frac{1}{t+1})q_n+q_{n-1}}\right|\\
\geq&\frac{|l_{n+1}-r_{n+1}+\frac{1}{t+1}-\frac{1}{t}|}{(M+2)^2q^{2}_n}\geq\frac{1}{2(M+2)^2q^{2}_n}\\
\geq&\frac{1}{(M+2)^3}|I_n(\sigma_1,\cdots,\sigma_n)|.
\end{align*}
\item If $n+2\neq m_k$ for any $k\geq1$, then we have
\begin{align*}
|y-x|=&\left|\frac{(r_{n+1}+1)p_n+p_{n-1}}{(r_{n+1}+1)q_n+q_{n-1}}-
\frac{(l_{n+1}+\frac{1}{M+1})p_n+p_{n-1}}{(l_{n+1}+\frac{1}{M+1})q_n+q_{n-1}}\right|\\
\geq&\frac{|r_{n+1}+1-l_{n+1}-\frac{1}{M+1}|}{(M+2)^2q^{2}_n}\geq\frac{1}{(M+2)^3q^{2}_n}\\
\geq&\frac{1}{(M+2)^3}|I_n(\sigma_1,\cdots,\sigma_n)|.
\end{align*}
\end{enumerate}
\end{proof}
\indent Consider a map $f:\ E_{M}(\psi)\rightarrow E_{M}$ defined as follows: for any $x=[\sigma_1,\cdots,\sigma_n]$, let
\[f(x)=\lim\limits_{n\to\infty}\overline{[\sigma_1,\cdots,\sigma_n]}.\]
For any $\varepsilon>0$, when $x,y\in E_{M}(\psi)$ satisfying
\begin{equation*}
|x-y|<\frac{1}{(M+2)^3}\min\limits_{(\sigma_1,\cdots,\sigma_{N_1})\in C_{N_1}}\left\{I_{N_1}(\sigma_1,\cdots,\sigma_{N_1})\right\},
\end{equation*}
where $N_1$ is the same as in Lemma \ref{qjbj}. It follows Lemma \ref{qjbj} and Lemma \ref{jg} that
\[|f(x)-f(y)|\leq|\overline{I}_{n}(\sigma_1,\cdots,\sigma_{n})|
\leq|I_{n}(\sigma_1,\cdots,\sigma_{n})|^{\frac{1}{1+\varepsilon}}
\leq(M+2)^{\frac{3}{1+\varepsilon}}|x-y|^{\frac{1}{1+\varepsilon}}.\]
Hence by \cite[Proposition 2.3]{Fal90} and Lemma \ref{phibh}, we have
\[
\dim_{\rm H}E_{\sup}(\psi)\geq\dim_{\rm H}E_{M}(\psi)\geq\frac{1}{1+\varepsilon}\dim_{\rm H}E_{M}.
\]
Let $\varepsilon\rightarrow0^{+}$ and then $M\rightarrow\infty$, we have $\dim_{\rm H} E_{\sup}(\psi)=1$. The proof is complete.\\
\subsubsection{\textbf{The case} $\psi(n)/n\to\alpha\ (0<\alpha<\infty)$ as $n\to\infty$.}
In this case, it is easy to see that  the set $E_{\sup}(\psi)$ equals to
\[F(\alpha):=\left\{x\in[0,1):\ \limsup\limits_{n\to\infty}\frac{\log a_n(x)}{n}=\alpha\right\}.\]
Hence, in order to get $\dim_{\rm H}E_{\sup}(\psi)$, it is equivalent to obtain $\dim_{\rm H}F(\alpha)$.\\
\textbf{Upper bound:} For $x\in F(\alpha)$, for any $0<\varepsilon<\alpha$, there are infinitely many $n'$s such that
\[a_n(x)\geq\exp(n(\alpha-\varepsilon)).\]
It then follows that
\[ F(\alpha)\subseteq\left\{x\in[0,1):\ a_n(x)\geq(\exp(\alpha-\varepsilon))^{n}\ \text{for infinitely many}\ n\right\}.\]
Let $\varepsilon\to0^{+}$, we deduce from \cite[Lemma 2.6, Theorem 3.1]{WW08} that
\[\dim_{\rm H}F(\alpha)\leq S(\alpha).\]
\textbf{Lower bound:} We can use the same method in the proof of \cite[Theorem 3.1]{WW08}. For completeness, here we just outline the proof. Fix $M\in\mathbb{N}$, let $\{m_k\}_{k\geq1}\subseteq\mathbb{N}$ be a subsequence satisfying $m_1=1$ and $m_1+\cdots+m_k\leq\frac{1}{k+1}m_{k+1}$ for all $k\geq2$. Define
\begin{eqnarray*}
  F_{M}(\exp(\alpha))=\{x\in[0,1):\ \lfloor\exp (\alpha m_k)\rfloor+1\leq a_{m_k}(x)\leq2\lfloor\exp (\alpha m_k)\rfloor\ \text{for all} \ \ k\geq 1 \
   \text{and} \\
   1\leq a_i(x)\leq M\ \text{for}\ i\neq m_k\ \text{for any}\ k\geq1\}.
\end{eqnarray*}
It is easy to verify that $F_{M}(\exp(\alpha))\subseteq F(\alpha)$. Let $M\to\infty$, it follows \cite[Lemma 2.5]{WW08} that
\[\dim_{\rm H}F(\alpha)\geq S(\alpha).\]
\subsubsection{\textbf{The case} $\psi(n)/n\to\infty$ as $n\to\infty$.}
\textbf{Upper bound:} For $x\in E_{\sup}(\psi)$, for any $\varepsilon>0$, there exist infinitely many $n'$s such that
\[\log a_n(x)\geq\psi(n)(1-\varepsilon).\]
Thus we have $E_{\sup}(\psi)\subseteq\hat{E}(\psi)$, where
\[
\hat{E}(\psi)=\left\{x\in[0,1):\ a_{n}(x)\geq\exp(\psi(n)(1-\varepsilon))\ \text{for infinitely many}\ n\right\}.
\]
Notice that $\psi(n)/n \to \infty$ as $n \to \infty$, then it follows \eqref{phitj} and \cite[Theorem 4.2]{WW08} that
\[\dim_{\rm H}E_{\sup}(\psi)\leq\dim_{\rm H}\hat{E}(\psi)=\frac{1}{A+1}.\]

\textbf{Lower bound:} We construct a suitable Cantor type subset of $E_{\sup}(\psi)$ in three steps. Firstly we define function
\begin{equation}\label{phs}
\theta(n)=\min\limits_{k\geq n}\psi(k)\ \ \text{for all}\ n\in\mathbb{N}.
\end{equation}
So $\theta(n)$ is well defined for $\psi(n)\to\infty$. It follows \eqref{phs} that
\begin{equation}\label{psp}
\theta(n)\leq\psi(n)\ \ \text{and}\ \ \theta(n)\leq \theta(n+1)\ \ \text{for any}\ n\geq1.
\end{equation}
Furthermore, we claim that
\begin{equation}\label{ts}
\theta(n)=\psi(n)\ \ \text{for infinitely many}\ n.
\end{equation}
If not, there exist $N\in\mathbb{N}$ such that for any $n\geq N$, $\theta(n)<\psi(n)$. In view of \eqref{psp},
\[\theta(n)<\min\limits_{k\geq n}\psi(k)\ \ \ \text{for any}\ n\geq N,\]
which contradicts to \eqref{phs}. Secondly, we define a sequence $\{d_n\}_{n\geq1}$ as follows:
\begin{equation}\label{an}
d_1=\exp \theta(1)\ \ \text{and}\ \ d_n=\min\left\{\exp \theta(n),\prod\limits_{k=1}^{n-1}d^{A-1+\varepsilon}_k\right\}\ (n\geq2).
\end{equation}
Then we easily obtain
\begin{equation}\label{an3}
\limsup\limits_{n\to\infty}\frac{\log d_{n+1}}{\log d_1+\cdots+\log d_n}\leq\limsup\limits_{n\to\infty}
\frac{\log\left(\prod\limits_{k=1}^{n-1}d^{A-1+\varepsilon}_k\right)}{\log\prod\limits_{k=1}^{n-1}d_k}=A-1+\varepsilon.
\end{equation}
By \eqref{an}, we also claim that
\begin{equation}\label{an4}
d_n=\exp \theta(n)\ \ \text{for infinitely many}\ n.
\end{equation}
If not, then there exist $N\in\mathbb{N}$ such that for any $n\geq N$,
\begin{equation}\label{an5}
\exp \theta(n)>d_n\ \ \ \text{and}\ \ d_n=\prod\limits_{k=1}^{n-1}d^{A-1+\varepsilon}_k,
\end{equation}
It follows \eqref{an5} that $d_n=d^{A+\varepsilon}_{n-1}\ (n\geq3)$ and
\begin{align}\label{zhj}
\nonumber \ &\prod\limits_{k=1}^{n}d_k=(\prod\limits_{k=1}^{N}d_k)\cdot d_{N+1}\cdot d_{N+2}\cdots d_{n}\\
\nonumber=&(\prod\limits_{k=1}^{N}d_k)\cdot(\prod\limits_{k=1}^{N}d_k)^{A-1+\varepsilon}\cdot
(\prod\limits_{k=1}^{N}d_k)^{(A-1+\varepsilon)(A+\varepsilon)}\cdots
(\prod\limits_{k=1}^{N}d_k)^{(A-1+\varepsilon)(A+\varepsilon)^{n-N-1}}\\
=&(\prod\limits_{k=1}^{N}d_k)^{1+(A-1+\varepsilon)+(A-1+\varepsilon)(A+\varepsilon)+\cdots(A-1+\varepsilon)(A+\varepsilon)^{n-N-1}}
=(\prod\limits_{k=1}^{N}d_k)^{(A+\varepsilon)^{n-N}}.
\end{align}
Combining \eqref{an5} with \eqref{zhj}, we have
\begin{equation}\label{bds}
\frac{\log \theta(n+1)}{n+1}>\frac{n+1-N}{n+1}\log(A+\varepsilon)+
\frac{\log\left(\frac{A-1+\varepsilon}{A+\varepsilon}\sum_{k=1}^{N}\log d_k\right)}{n+1},
\end{equation}
then by \eqref{psp} and \eqref{bds},
\[\liminf\limits_{n\to\infty}\frac{\log\psi(n+1)}{n+1}\geq\liminf\limits_{n\to\infty}\frac{\log \theta(n+1)}{n+1}\geq\log(A+\varepsilon)>\log A,\]
which contradicts to \eqref{phitj}. Thirdly, we chaim that
\begin{equation}\label{dn}
d_n\to\infty\ \ \text{as}\ \  n\to\infty.
\end{equation}
In fact, on one hand, by \eqref{psp} and \eqref{an}, we have
\begin{equation}\label{an2}
\limsup\limits_{n\to\infty}\frac{\log d_n}{\psi(n)}\leq1\ \text{and}\ \ d_n\leq d_{n+1}\ (n\geq2).
\end{equation}

On the other hand, by \eqref{ts} and \eqref{an4}, we obtain
\begin{equation}\label{psp3}
\limsup\limits_{n\to\infty}\frac{\log d_n}{\psi(n)}\geq1.
\end{equation}
 Now we use the sequence $\{d_n\}_{n\geq1}$ to construct the subset of $E(\psi)$.
Let $M$ be the positive integer such that $M\lfloor d_n\rfloor\geq3$ for all $n \geq 1$. Define
\[E(\{d_n\}_{n\geq1})=\{x\in[0,1):\ M\lfloor d_n\rfloor\leq a_n(x)<2M\lfloor d_n\rfloor,\ \forall\ n\geq1\}.\]
It follows \eqref{an2} and \eqref{psp3} that
\[E(\{d_n\}_{n\geq1})\subseteq E_{\sup}(\psi).\]
Let $\varepsilon\to 0^{+}$, by \eqref{an3} and Lemma \ref{flww}, we have
\[\dim_{\rm H}E_{\sup}(\psi)\geq\dim_{\rm H}E(\{d_n\}_{n\geq1})=\frac{1}{2+\limsup\limits_{n\to\infty}\frac{\log d_{n+1}}{\log d_1+\cdots+\log d_n}}\geq\frac{1}{A+1}.\]
\subsection{The proof of Theorem \ref{jxybqk}}
  \textbf{Upper bound:}
  For any $0<\varepsilon<1$, we have $E(\psi)\subseteq\bigcup_{N=1}^{\infty}E_{\psi}(N)$, where
  \[E_{\psi}(N)=
 \Big\{x\in [0,1):\ \exp((1-\varepsilon)\psi(n))\leq a_n(x)\leq\exp((1+\varepsilon)\psi(n)),\ \forall\ n\geq N\Big\}.\]

  Then we obtain
 \begin{equation}\label{3bhgx}
 \dim_{\rm H}E(\psi)
 \leq\sup\limits_{N\geq1}\{\dim_{\rm H}E_{\psi}(N)\}.
 \end{equation}
 The following we only consider  the upper bound Hausdorff dimension of $E_{\psi}(1)$ since the proof for other cases is similar. For any $n\geq 1$, let
 \begin{equation}\label{dne}
 D_{n}(\varepsilon)=\{(a_1,\cdots, a_n)\in \mathbb{N}^{n}: \exp\left((1-\varepsilon)\psi(n))
 \leq a_n(x)\leq\exp((1+\varepsilon)\psi(n)\right)\}.
 \end{equation}
 Then we deduce that
 \begin{equation}\label{3gs}
  \sharp D_{n}(\varepsilon)\leq \prod\limits_{k=1}^{n}\left(2\varepsilon\psi(k)\exp((1+\varepsilon)\psi(k))\right).
 \end{equation}
 Notice that for any $n\geq1$,
 \[E_{\psi}(1)\subseteq\bigcup\limits_{(\sigma_1,\cdots,\sigma_n)\in D_{n}(\varepsilon)}J_{n}(\sigma_1,\cdots,\sigma_n),\]
 where
 \[J_{n}(\sigma_1,\cdots,\sigma_n)=\bigcup\limits_{\sigma_{n+1}}I_{n+1}(\sigma_1,\cdots,\sigma_n,\sigma_{n+1}),\]
 where the union is taken over all $\sigma_{n+1}$ such that $(\sigma_1,\cdots,\sigma_n,\sigma_{n+1})\in D_{n+1}(\varepsilon)$. Thus we get a cover of $E_{\psi}(1), i.e., \{J_n(\sigma_1,\cdots,\sigma_n):\ (\sigma_1,\cdots,\sigma_n)\in D_{n}(\varepsilon)\}$. For any $(\sigma_1,\cdots,\sigma_n)\in D_{n}(\varepsilon)$,
 \begin{eqnarray}\label{4cd}
 \nonumber |J_{n}(\sigma_1,\cdots,\sigma_n)|&\leq& \sum\limits_{\sigma_{n+1}\geq \exp((1-\varepsilon)\psi(n+1))}|I_{n+1}(\sigma_1,\cdots, \sigma_n, \sigma_{n+1})|\\
 \nonumber&\leq& \left(\sum\limits_{\sigma_{n+1}\geq \exp((1-\varepsilon)\psi(n+1))}\frac{1}{\sigma^{2}_{n+1}}\right)\cdot
 \exp\left(-2(1-\varepsilon)\sum\limits_{k=1}^{n}\psi(k)\right)\\
 &\leq&\exp\left(-(1-\varepsilon)\psi(n+1)-2(1-\varepsilon)\cdot\sum\limits_{k=1}^{n}\psi(k)\right).
 \end{eqnarray}
 Combining \eqref{3gs} with \eqref{4cd}, we have
 \[\dim_{\rm H}E_{\psi}(1)\leq\liminf\limits_{n\to\infty}\frac{\log\sharp D_n(\varepsilon)}{-\log|J_{n}(\sigma_1,\cdots,\sigma_n)|}\leq\frac{1+\varepsilon}{1-\varepsilon}\cdot
 \frac{1}{2+\limsup\limits_{n\to\infty}\frac{\psi(n+1)}{\psi(1)+\cdots+\psi(n)}}.\]
 Let $\varepsilon\to0^{+}$, together with \eqref{3bhgx}, we obtain
 \[\dim_{\rm H}E(\psi)\leq\frac{1}{2+\limsup\limits_{n\to\infty}\frac{\psi(n+1)}{\psi(1)+\cdots+\psi(n)}}.\]
 \textbf{Lower bound:} Let $M$ be the positive integer such that $M \lfloor \exp\psi(n)\rfloor\geq3$ for any $n\geq1$. Define
\[\mathcal{E}(\psi)=\Big\{x\in[0,1): M \lfloor \exp\psi(n)\rfloor\leq a_n(x)<2M \lfloor \exp\psi(n)\rfloor, \forall\ n\geq1\Big\}.\]
It is easy to verify that
\[\mathcal{E}(\psi)\subseteq E(\psi).\]
 By Lemma \ref{flww}, We conclude that  \[\dim_{\rm H}E(\psi)\geq\dim_{\rm H}\mathcal{E}(\psi)=\frac{1}{2+\limsup\limits_{n\to\infty}\frac{\psi(n+1)}{\psi(1)+\cdots+\psi(n)}}.\]
\subsection{The proof of Theorem \ref{xjx}}
\textbf{ Upper bound:} We can use the same proof in \cite[Theorem 2.8]{FS19}, then
\[\dim_{\rm H}E_{\inf}(\psi)\leq\frac{1}{B+1}.\]
\textbf{ Lower bound:} It is trivial for $B=\infty$ and thus we always assume that $1 \leq B <\infty$.
Since \[\limsup\limits_{n\to\infty}\frac{\log\psi(n)}{n}=\log B,\]
for any $\varepsilon>0$, we have $\psi(n) \leq (B+\varepsilon/2)^n$ for $n$ large enough. This implies
\begin{equation}\label{qy0}
 \psi(n)(B+\varepsilon)^{j-n} \leq (B+\varepsilon/2)^n(B+\varepsilon)^{j-n} \to 0 \ \ (n \to \infty).
\end{equation}
We define a sequence $\{L_j\}_{j\geq1}$ as follows:
\begin{equation}\label{tidy}
 L_j=\sup\limits_{n\geq j}\{\exp\left(\psi(n)(B+\varepsilon)^{j-n}\right)\},\ j=1,2,\cdots.
\end{equation}
 It is easy to check that $L_{j+1}\leq L^{B+\varepsilon}_{j}$ and then we deduce that
 \begin{equation}\label{tigx}
 \log L_{n+1}-\log L_1\leq(B+\varepsilon-1)\sum\limits_{j=1}^{n}\log L_j.
 \end{equation}
 In view of \eqref{tidy}, we claim that
\begin{equation}\label{tixjx}
 \liminf\limits_{n\to\infty}\frac{\log L_n}{\psi(n)}=1.
 \end{equation}
Indeed, on one hand, we get $L_j \geq\exp\psi(j)$ for all $j \geq 1$ by the definition of $L_j$, and hence
\[
\liminf\limits_{n\to\infty}\frac{\log L_n}{\psi(n)}\geq1.
\]
On the other hand, in view of \eqref{qy0}, the supremum in \eqref{tidy} is achieved, we denote by $t_j \geq j$ the smallest number for which
\[L_j = \exp\left(\psi(t_j)(B+\varepsilon)^{j-t_j}\right).\] Observe that for many consecutive j's, the number $t_j$ will be the same. More precisely, $t_j=t_{j+1}=\cdots=t_{t_j}$.
Let $\{\ell_k\}_{k\geq1}$ be a strictly increasing sequence of $\{t_j\}_{j\geq1}$. Then we obtain $L_{\ell_k} = \exp\psi(\ell_k)$ and thus
\[
\liminf\limits_{n\to\infty}\frac{\log L_n}{\psi(n)}\leq  \liminf\limits_{k\to\infty}\frac{\log L_{\ell_k} }{\psi(\ell_k)}  =1.
\]
Now we use the sequence $\{L_j\}_{j\geq1}$ to construct the subset of $E_{\inf}(\psi)$. Let $M$ be the positive integer such that $M\lfloor L_n\rfloor\geq3$ for all $n \geq 1$. Define
\[
E_{\inf}(\{L_n\}_{n\geq1})=\big\{x\in [0,1): M\lfloor L_n\rfloor\leq a_n(x)<2M\lfloor L_n\rfloor, \forall\ n\geq1\big\}.
\]
It is easy to verify that
\[E_{\inf}(\{L_n\}_{n\geq1})\subseteq E_{\inf}(\psi).\]
Let $\varepsilon\to 0^{+}$, combining \eqref{tigx} with Lemma \ref{flww}, we have
\[\dim_{\rm H}E_{\inf}(\psi)\geq\dim_{\rm H}E_{\inf}(\{L_n\}_{n\geq1})=\frac{1}{2+\limsup\limits_{n\to\infty}\frac{\log L_{n+1}}{\log L_1+\cdots+\log L_n}}\geq\frac{1}{B+1}.\]

\textbf{Acknowledgement:} The authors would like to thank Professor Lingmin Liao for his invaluable comments. This research was supported by National Natural Science Foundation of China (11771153, 11801591, 11971195) and Fundamental Research Funds for the Central Universities SYSU-18lgpy65. Kunkun Song would like to thank China Scholarship Council (CSC) for ﬁnancial support (201806270091).

\end{document}